\documentclass{amsart}

\usepackage{graphicx}
\usepackage[all]{xy}
\usepackage{color}
\usepackage{amsmath}
\usepackage{amsfonts}
\usepackage{amsthm}
\usepackage{amssymb}
\usepackage{mathrsfs}
\usepackage{hyperref}

\theoremstyle{theorem}
\newtheorem{thrm}{Theorem}
\newtheorem{lem}[thrm]{Lemma}
\newtheorem{cor}[thrm]{Corollary}
\newtheorem{prop}[thrm]{Proposition}

\numberwithin{thrm}{section} 

\def \Dj{\mbox{\raise0.3   ex\hbox{-}\kern-0.4em D}}

\theoremstyle{remark}
\newtheorem{rem}[thrm]{Remark}

\def\R{\mathbb{R}}

\def\Z{\mathbb{Z}}

\def\C{\mathbb{C}}

\begin{document}

\title[Harmonic functions with highly intersecting zero sets]{Harmonic functions with highly intersecting zero sets}

\author[Vuka\v sin Stojisavljevi\'c ]{Vuka\v sin Stojisavljevi\'c}
\address{D\'e\-par\-te\-ment de math\'ematiques et de
sta\-tistique, Univer\-sit\'e de Mont\-r\'eal,  CP 6128 succ
Centre-Ville, Mont\-r\'eal,  QC  H3C 3J7, Canada.}
\email{vukasin.stojisavljevic@gmail.com}
\maketitle

\begin{abstract}
We show that the number of isolated zeros of a harmonic map $h:\R^2\to \R^2$ inside the ball of radius $r$ can grow arbitrarily fast with $r$, while its maximal modulus grows in a controlled manner. This result is an analogue, in the context of harmonic maps, of the celebrated Cornalba-Shiffman counterexamples to the transcendental B\'ezout problem.
\end{abstract}


\section{Introduction}

\subsection{The results}

For $r\geq 0$, let $B_r\subset \R^d$ be the closed ball of radius $r.$ Given $h:\R^d \to \R^d,$ we denote by $\mu(h,r)=\max_{ B_r} |h|$ and by  $\zeta(h,r)$ the number of connected components of $h^{-1}(0) \cap B_r.$ It readily follows that
$$\text{ the number of isolated zeros of } h \text{ in } B_r \leq \zeta(h,r).$$
We use the same notation for complex spaces and $f:\C^d\to \C^d$, by identifying $\C^d\cong\R^{2d}.$

Recall that $h=(h_1,h_2):\R^2 \to \R^2$ is called \textit{harmonic} if both $h_1$ and $h_2$ are harmonic.  Note that $h^{-1}(0)=h_1^{-1}(0)\cap h_2^{-1}(0)$, i.e. the zeros of $h$ correspond to the intersections of zero sets of $h_1$ and $h_2.$ The following is the main result of the paper:

\begin{thrm}\label{Thm:Main} Let $\varepsilon>0$ and let $\{n_i\}_{i\geq 1}$ be any sequence of positive integers. There exists a harmonic map $h:\R^2\to \R^2$, such that $\mu(h,r)=O(e^{r^{1+\varepsilon}})$ as $r\to \infty$, and for all integers $k\geq 1$, $h$ has at least $\sum_{i=1}^k n_i$ isolated zeros inside $B_{2^k}.$ In particular, $\zeta(h,r)$ can grow arbitrarily fast with $r$,  for a suitable choice of $\{n_i\}.$
\end{thrm}



By considering the holomorphic extension of $h$ to $\C^2$, we obtain the following corollary of Theorem \ref{Thm:Main}. 

\begin{cor}\label{Coro:Holomorphic}
Let $\varepsilon>0$ and let $\{n_i\}_{i\geq 1}$ be any sequence of positive integers. There exists a holomorphic map $f:\C^2\to \C^2$, such that $\mu(f,r)=O(e^{r^{1+\varepsilon}})$ as $r\to \infty$, and for all integers $k\geq 1$, $f$ has at least $\sum_{i=1}^k n_i$ isolated zeros inside $B_{2^k}.$ In particular, $\zeta(f,r)$ can grow arbitrarily fast with $r$,  for a suitable choice of $\{n_i\}.$
\end{cor}

\begin{rem}
Theorem \ref{Thm:Main} and Corollary \ref{Coro:Holomorphic} readily extend to higher dimensions. Indeed, let $d\geq 3$ and denote the coordinates on $\R^d$ by $(x_1,\ldots, x_d).$ We have that $\tilde{h}:\R^d\to \R^d$, given by $\tilde{h}(x_1,\ldots,x_d)=(h(x_1,x_2), x_3,\ldots,x_d)$, has the same properties as $h.$ The same holds for $\tilde{f}:\C^d \to \C^d, \tilde{f}(z_1,\ldots,z_d)=(f(z_1,z_2), z_3,\ldots,z_d).$
\end{rem}

The inspiration for Theorem \ref{Thm:Main} comes from the examples of Cornalba and Shiffman - \cite{CornalbaShiffman}. These examples were used to disprove a prediction about an upper bound on $\zeta(f,r)$ in terms of $\mu(f,r)$ when $f$ is holomorphic. Corollary \ref{Coro:Holomorphic} gives another set of counterexamples to this prediction.  In this regard, Theorem \ref{Thm:Main} can be considered as an analogue of Cornalba-Shiffman examples in the realm of harmonic maps.  We now provide some background about these examples and put Theorem \ref{Thm:Main} and Corollary \ref{Coro:Holomorphic}  in broader context.

\subsection{Context} 

Given a holomorphic map $f:\C^d\to \C^d$, the idea of bounding $\zeta(f,r)$ in terms of $\mu(f,r)$ has been proposed by Griffiths in the 1970s under the name \textit{transcendental B\'ezout problem}, see \cite{Griffiths71}. Philosophical roots of this idea go back to Serre’s famous G.A.G.A. program - \cite{SerreGAGA}. For holomorphic maps, $\log \mu(f,r)$ plays the role of the degree of $f.$ In the case $d=1$, a bound on $\zeta(f,r)$ in terms of $\log \mu(f,r)$ is given by the Jensen's formula. The existence of a similar bound in higher dimensions was predicted as a part of the transcendental B\'ezout problem. This prediction was disproved by Cornalba and Shiffman in \cite{CornalbaShiffman}.  They constructed, for every $\varepsilon>0$, a holomorphic map $f:\C^2 \to \C^2$ such that $\mu(f,r)= O(e^{r^{\varepsilon}})$, while $\zeta (f,r)$ can grow arbitrarily fast. Corollary \ref{Coro:Holomorphic} gives another set of counterexamples to the conjectured transcendental B\'ezout bound.

In the opposite direction, positive results have been obtained for certain modifications of the original transcendental B\'ezout problem, see \cite{Stoll72,Carlson76,Gruman77,BY86, Ji95, LT96, Li98} and references therein. Most recently, a \textit{coarse viewpoint} on the problem, inspired by topological persistence, has been introduced in \cite{BPPSS23}.  From this point of view,  holomorphic and harmonic maps are considered on equal footing. In the case of holomorphic maps, the results of \cite{BPPSS23} are juxtaposed with the Cornalba-Shiffman examples or Corollary \ref{Coro:Holomorphic} and in the case of harmonic maps with Theorem \ref{Thm:Main}. We briefly recall the relevant results from \cite{BPPSS23}.

Let $F:\R^d \to \R^d$ be an arbitrary map. For $\delta, r \geq 0$,  \textit{the coarse count of zeros of} $F$, denoted by $\zeta(F,r,\delta)$, is the number of connected components of $\{ |F| \leq \delta \} \cap B_r$ which contain zeros of $F.$ Note that $\zeta(F,r,0)=\zeta(F,r).$ The same definition extends to complex spaces via $\C^d\cong \R^{2d}.$

\begin{thrm}[\cite{BPPSS23}]\label{Theorem_coarse_counts}
Let $a>1.$ For any holomorphic map $f: \C^d \to \C^d$,  $r>0$ and $\delta \in \left( 0, \frac{\mu(f, ar)}{e} \right)$,  it holds
\begin{equation}
\label{eq:coarse_holomorphic}
\zeta(f,r,\delta) \leq C_0 \left(\log\left(\frac{\mu(f, ar)}{\delta}\right)\right)^{2d-1},
\end{equation}
where the constant $C_0$ depends only on $a$ and $d$. 

On the other hand, for any harmonic map $h:\R^d \to \R^d$ there exists $C_1 \geq 1$, which depends only on $d$, such that for $r >0$ and $\delta \in \left( 0, \frac{\mu(h,C_1 ar)}{e} \right)$, it holds
\begin{equation}\label{eq:coarse_harmonic} 
\zeta(h,r,\delta) \leq C_2 \left(\log\left(\frac{\mu(h, C_1 ar)}{\delta}\right)\right)^{d}, 
\end{equation}
where $C_2$ depends only on $a$ and $d.$
\end{thrm}

Note that the right-hand sides of (\ref{eq:coarse_holomorphic}) and (\ref{eq:coarse_harmonic}) blow up as $\delta \to 0.$ In other words, Theorem \ref{Theorem_coarse_counts} does not provide a bound for the genuine count of zeros. This is expected since Cornalba-Shiffman examples and Corollary \ref{Coro:Holomorphic} show that such a bound cannot exist for holomorphic maps in general, while Theorem \ref{Thm:Main} demonstrates the same thing in the case of harmonic maps.

\begin{rem}
In \cite{BPPSS23}, a complete proof of (\ref{eq:coarse_holomorphic}) was given, while a proof of (\ref{eq:coarse_harmonic}) was outlined along the same lines. A detailed proof of (\ref{eq:coarse_harmonic}) will appear in \cite{LMH25}.
\end{rem}

\begin{rem}
The sharp contrast between classical and coarse counting has been observed in  a number of recent works, addressing seemingly unrelated questions. We already mentioned the upper bounds on the coarse counts of zeros of holomorphic and harmonic maps, which are contrasted with Cornalba-Shiffman examples, Corollary \ref{Coro:Holomorphic} and Theorem \ref{Thm:Main}. On the other hand,  an upper bound on the coarse count of nodal domains of linear combinations of eigenfunctions of elliptic operators has been proven in \cite{BPPPSS22}. The fact that such a bound does not exist for the genuine count was previously shown in \cite{BLS20}. Lastly, it was shown in \cite{CGG21} that the coarse count of periodic orbits of a Hamiltonian diffeomorphism of a monotone symplectic manifold grows at most exponentially with iterations. The non-existence of such a bound for the genuine count is a well-studied phenomenon in dynamics, see \cite{Asaoka17} and references therein.
\end{rem}

\section*{Acknowledgments}

I cordially thank Sasha Logunov for very inspiring and useful discussions, as well as for a number of comments on the preliminary versions of the draft. His help in carrying out this project has been truly indispensable. I thank Lev Buhovsky for valuable remarks. I was supported by CRM-ISM postdoctoral fellowship and Fondation Courtois.

\section{Proof of Theorem \ref{Thm:Main}}

\subsection{Sketch of the proof}

Let $\{c_k \geq 2\}_{k\geq 1}$ be a sequence of integers and 
$$u_k(x,y)=\sum_{j=1}^{2c_k-1}b_{k,j}\sin \left(\frac{\pi}{2^k} jx \right) e^{\frac{\pi}{2^k}jy},$$
where coefficients $b_{k,j}$ will be specified later. One readily checks that $u_k$ are harmonic and we define
$$g(x,y)=\sum_{k=1}^\infty a_k u_k.$$
Firstly, we claim that by taking $a_k>0$ which decreases sufficiently fast with $k$, we can obtain that the above power series converges uniformly on compact sets, implying that $g$ is well-defined and harmonic. Furthermore, we can guarantee that $\mu(g,r)=O(e^{r^{1+\varepsilon}}).$ These considerations are formalized in Proposition \ref{prop:modulus}.

Now, we define the desired harmonic map $h:\R^2\to \R^2$ as
$$h(x,y)=(g(x,y), \sin (\pi x) e^{\pi y}).$$
We claim that $\zeta(h,r)$ can grow arbitrarily fast if $\{c_k\}_{k\geq 1}$ grows sufficiently fast. To this end, notice that $\sin (\pi x) e^{\pi y}=0$ if and only if $x\in \Z,$ thus
$$h(x,y)=0 \Leftrightarrow g(x,y)=0 \text{ and } x\in \Z.$$
Hence, we need to count zeros of $g$ on vertical lines $x\in \Z.$ We start by analyzing zeros of $u_k$ for a fixed $k\geq 1.$

First, we notice that $u_k=0$ on lines $\{x=2^l\}$, $l\geq k$ since $\sin \left(\frac{\pi}{2^k} j 2^l \right) =0.$ On the other hand, for $x=2^{k-1}$, we have that
$$u_k(2^{k-1},y) = \sum_{j=1}^{2c_k-1}b_{k,j}\sin \left(\frac{\pi}{2} j \right) (e^{\frac{\pi}{2^k}y})^j=P(e^{\frac{\pi}{2^k}y}),$$
where $P$ is a polynomial of degree $2c_k-1.$ By suitably choosing $b_{k,j}$, we can arrange that this polynomial has $c_k-1$ distinct zeros in $(0,1)$, which implies that $u_k(2^{k-1},y)$ has $c_k-1$ zeros with $y\in [-2^{k-1},0].$ These considerations are formalized in Lemma \ref{Lem:Zeros_u}. The zero set of $u_k$ is sketched on Figure \ref{Figure:Zeros}.

\begin{figure}[ht]
	\begin{center}
		\includegraphics[scale=0.77]{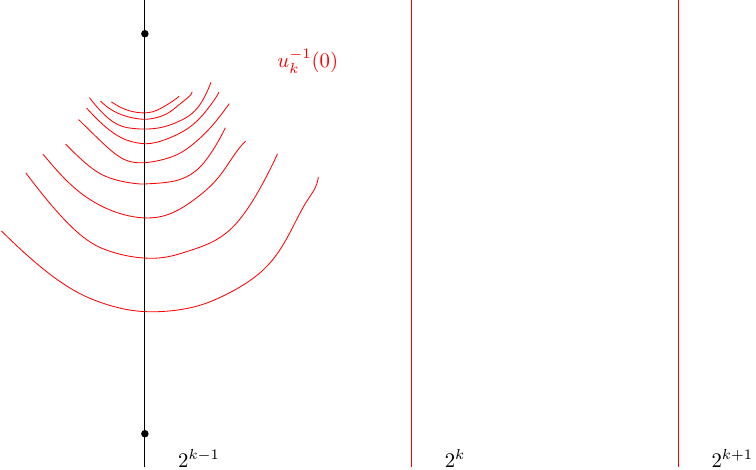}
		\caption{Zero set of $u_k.$}
		\label{Figure:Zeros}
	\end{center}
\end{figure}

The last step of the proof is the most delicate. It consists of inductively choosing $a_k$ in such a way that Figure \ref{Figure:Zeros} is transplanted on each vertical line $\{ x=2^k\}_{k\geq 1}.$ More precisely, we first fix $a_1>0$ and notice that $g_1=a_1 u_1$ has $c_1-1$ zeros on $\{x=1\},$ while $g_1=0$ on $\{x=2\},\{x=2^2 \},\{x=2^3\},...$ Now, we choose $a_2$ sufficiently small so that by adding $a_2u_2$ to $g_1$, we preserve the $c_1-1$ zeros on $\{x=1\},$ while introducing $c_2-1$ zeros on $\{x=2\}.$ In other words, $g_2=a_1 u_1+ a_2 u_2$ has $c_1-1$ zeros on $\{x=1\},$ $c_2-1$ zeros on $\{x=2\}$, and $g_2=0$ on $\{x=2^2\},\{x=2^3 \},...$ We proceed inductively to define $g$ which has exactly $c_k-1$ zeros on each $\{x=2^{k-1} \}, k\geq 1.$ The induction is detailed in Proposition \ref{prop:zeros}.

\subsection{A detailed proof}
For an integer $c\geq 2$ let 
$$P_c(x)=x\left( x^2-1+\frac{1}{2} \right)\left(x^2-1+\frac{1}{3} \right)\cdot \ldots \cdot \left(x^2-1+\frac{1}{c}\right).$$
Given a sequence of integers $\{c_k \geq 2\}_{k\geq 1}$, we define $\{b_{k,j}\}_{k\geq 1, 1\leq j \leq 2c_k -1}$ by the following identities
$$\sum_{j=1}^{2c_k-1} (-1)^{\frac{j-1}{2}} b_{k,j} x^j = x\left( x^2-1+\frac{1}{2} \right)\left(x^2-1+\frac{1}{3} \right)\cdot \ldots \cdot \left(x^2-1+\frac{1}{c_k}\right)=P_{c_k}(x).$$
Note that if $j$ is even the coefficient of $x^j$ in $P_c(x)$ is zero. Thus, $b_{k,j}=0$ for all even $j$ and the expression $(-1)^{\frac{j-1}{2}}$ is well-defined for odd $j.$ As above, for $k\geq 1$, let $u_k:\R^2 \to \R$ be given by
$$u_k(x,y)=\sum_{j=1}^{2c_k-1}b_{k,j}\sin \left(\frac{\pi}{2^k} jx \right) e^{\frac{\pi}{2^k}jy}.$$
One readily checks that $u_k$ are harmonic. For a sequence of positive real numbers $\{a_k\}_{k\geq 1}$, let
$$g=\sum_{k=1}^\infty a_k u_k.$$
The following proposition implies that for sufficiently small $\{a_k\}_{k\geq 1}$, $g$ is well-defined,  harmonic and satisfies the desired estimate on the modulus.

\begin{prop}\label{prop:modulus}
Let $\varepsilon>0.$ There exist positive numbers $A_1,A_2,\ldots$ which depend on $\{c_k \}_{k\geq 1}$ and $\varepsilon$, such that whenever $a_k\leq A_k$ for all $k\geq 1$, $g_N=\sum_{k=1}^N a_k u_k$ converges uniformly on compact sets and the function $g=\lim_{N\to \infty} g_N$ satisfies
\begin{equation}\label{eq:Max_Estimate_g}
\mu(g,r)\leq e^{r^{1+\varepsilon}},
\end{equation}
for all $r\geq 0.$
\end{prop}
\begin{proof}
We will show that there exist $A_1,A_2, \ldots$ such that if $a_k\leq A_k$ for all $k\geq 1$ then for all $r\geq 0$
$$\sum_{k=1}^\infty a_k \mu(u_k,r) \leq e^{r^{1+\varepsilon}}.$$
This will imply that for all $r\geq 0$ the series $\sum_{k=1}^\infty a_k u_k|_{B_r}$ absolutely converges in $\| \cdot \|_{L^\infty}$ (and thus $g_N$ converges uniformly on compact sets), as well as that (\ref{eq:Max_Estimate_g}) holds. To this end, we first  notice that from the definition of $b_{k,j}$ it follows that $|b_{k,j}|\leq \binom{c_k-1}{(j-1)/2} \leq 2^{c_k-1}.$ Using this inequality we estimate
$$\sum_{k=1}^\infty a_k \mu( u_k,r )\leq \sum_{k=1}^\infty a_k \sum_{j=1}^{2c_k-1} |b_{k,j}| e^{\frac{\pi}{2^k}jr}\leq \sum_{k=1}^\infty a_k 2^{c_k-1} \sum_{j=1}^{2c_k-1} e^{\frac{\pi}{2^k}jr}\leq $$
$$\leq\sum_{k=1}^\infty a_k 2^{c_k-1} \sum_{j=1}^{2c_k-1} e^{\frac{\pi}{2}jr} =\sum_{k=1}^\infty a_k 2^{c_k-1} \sum_{j=1}^{2c_k-1} (e^{\frac{\pi}{2}r})^j\leq \sum_{k=1}^\infty a_k 2^{c_k-1} (e^{\frac{\pi}{2}r}+1)^{2c_k-1} \leq$$
$$\leq \sum_{k=1}^\infty a_k 2^{c_k-1} (2e^{\frac{\pi}{2}r})^{2c_k-1}=\sum_{k=1}^\infty a_k 2^{3c_k-2} (e^{\frac{\pi}{2}r})^{2c_k-1}.$$
Define a sequence $\{ d_i \}_{i\geq 1}$ by setting $d_{2c_k-1}=a_k2^{3c_k-2}, k\geq 1$ and $d_i=0$ when $i\notin \{2c_1-1,2c_2-1,\ldots  \}.$ Previous estimates give us that
\begin{equation}\label{eq:power_series_estimate}
\sum_{k=1}^\infty a_k \mu(u_k,r) \leq \sum_{k=1}^\infty d_k e^{\frac{\pi}{2}kr}.
\end{equation}
Since $c_k$ are fixed, by taking $a_k$ small enough we can make $d_k$ arbitrarily small. On the other hand, as $e^{r^{1+\varepsilon}}$ grows faster than $e^{\frac{\pi}{2}kr}$ for any $k$, there exist $A_1,A_2,\ldots$ such that if $a_k\leq A_k$ then
$$d_k e^{\frac{\pi}{2}kr} \leq \frac{1}{2^k} e^{r^{1+\varepsilon}},$$ 
for all $r\geq 0$. Combining these inequalities with (\ref{eq:power_series_estimate}) finishes the proof.
\end{proof}

\begin{rem}
In fact, the proof of Proposition \ref{prop:modulus} shows that for any $F:(0,+\infty)\to (0,+\infty)$ such that $\limsup_{r\to +\infty} \frac{e^{kr}}{F(r)}<+\infty$ for all $k\geq 1$, there exist $A_1,A_2,\ldots$ such that if $a_k\leq A_k$ for all $k\geq 1$ then $\mu(g,r)\leq F(r)$ for all $r\geq 0.$
\end{rem}

Next, we analyze the zeros of $g.$ The required properties are summarized in the following proposition.

\begin{prop}\label{prop:zeros}
Let $\{c_k \geq 2 \}_{k\geq 1}$ be a sequence of integers and $\{A_i \}_{i\geq 1}$ a sequence of positive real numbers. There exists a sequence $\{ a_i \}_{\geq 1}, 0<a_i\leq A_i$, such that for every $k\geq1$, $g$ defined by $\{ a_i \}$ has exactly $c_k-1$ zeros inside the interval $\{2^{k-1} \} \times  (-2^{k-1} ,0).$ Moreover, at each of these zeros $\frac{\partial g}{\partial y}\neq 0.$
\end{prop}

In order to prove Proposition \ref{prop:zeros}, we will need an auxiliary lemma about the zero sets of $u_k.$ Denote by $\xi_{k,j}=(2^{k-1}, \frac{2^{k-1}}{\pi} \log (1-\frac{1}{j})),$ for $k\geq 1$ and $2\leq j \leq c_k.$

\begin{lem}\label{Lem:Zeros_u}
For each $k\geq 1$, $2\leq j \leq c_k$ it holds $u_k(\xi_{k,j})=0$, $\frac{\partial u_k}{\partial y} (\xi_{k,j})\neq 0$ and $u_k$ has no other zeros on the line $\{x=2^{k-1} \}.$ Moreover, for any integer $l\geq k$, $u_k=0$ on the whole line $\{x=2^l \}.$
\end{lem}
\begin{proof}
By setting $x=2^{k-1}$ we obtain
$$u_k(2^{k-1},y)=\sum_{j=1}^{2c_k-1} b_{k,j} \sin \left( \frac{\pi}{2}j \right)e^{\frac{\pi}{2^k}jy}=\sum_{j=1}^{2c_k-1} b_{k,j} (-1)^{\frac{j-1}{2}}e^{\frac{\pi}{2^k}jy}=P_{c_k}(e^{\frac{\pi}{2^k}y}),$$
as $b_{k,j}=0$ when $j$ is even. Since the zeros of $P_{c_k}$ are $\{0, \pm\sqrt{1-1/j}\}_{2\leq j \leq c_k}$ and $e^{\frac{\pi}{2^k}y}>0$, we get that $\xi_{k,j}$ are the only zeros of $u_k$ on the line $\{x=2^{k-1} \}$.  $\frac{\partial u_k}{\partial y}(\xi_{k,j})\neq 0$ follows from the fact that $P_{c_k}$ has  $2c_k-1=\deg P_{c_k}$ distinct real zeros. In order to prove the moreover part it is enough to notice that for $x=2^l$, $\sin (\frac{\pi}{2^k}jx)=\sin(2^{l-k} \pi j ) =0$ for all $j.$
\end{proof}

\begin{proof}[Proof of Proposition \ref{prop:zeros}]

For $D\subset \R^2$ and a smooth function $F:\R^2\to \R$, we denote
$$\| F \|_{C^1(D)}= \max_{D} \Big\{ |F|, \Big | \frac{\partial F}{\partial x} \Big|,  \Big | \frac{\partial F}{\partial y} \Big| \Big\}.$$

Without loss of generality, we may assume that $\{A_i \}_{i\geq 1}$ is such that if for all $i\geq 1$, $0<a_i \leq A_i$ then $g$ is well-defined, see Proposition \ref{prop:modulus}. We will define the sequence $\{a_i \}_{i\geq 1}$ inductively.  In the $k$-th step, apart from $a_k$, we will also define a perturbation threshold $m_k$ which will be used to further define $a_i$, for $i\geq k+1.$

First, we choose $0<a_1 \leq A_1$ arbitrarily. By Lemma \ref{Lem:Zeros_u}, $g_1=a_1 u_1$ has exactly $c_1-1$ zeros in $\{1 \} \times [-1,0]$, which all belong to $\{1 \} \times (-1,0)$, and $\frac{\partial g_1}{\partial y}=a_1 \frac{\partial u_1}{\partial y}\neq 0$ at each of these zeros. Thus, there exists $m_1>0$ such that for every $p:\R^2\to \R$ which satisfies $\|p \|_{C^1(B_2)}<m_1$, the perturbation $g_1+p$ still has exactly $c_1-1$ zeros in $\{1 \} \times [-1,0],$ which all belong to $\{1 \} \times (-1,0)$ and at each of them $\frac{\partial (g_1+p)}{\partial y}\neq 0$. This is the desired threshold $m_1.$

Notice that $m_1$ has been chosen so that perturbations smaller than $m_1$ do not destroy the $c_1-1$ zeros we created by picking $a_1.$ More generally, assume that we picked $a_1,\ldots , a_k>0$ for some integer $k\geq 2.$ Then $g_k=\sum_{i=1}^k a_i u_i$ has exactly $c_k-1$ zeros in $\{ 2^{k-1} \} \times [-2^{k-1},0]$, which all belong to $\{ 2^{k-1} \} \times (-2^{k-1},0)$, and at each of them $\frac{\partial g_k}{\partial y}\neq 0.$ Indeed, by Lemma \ref{Lem:Zeros_u}, $u_1=u_2=\ldots = u_{k-1}=0$ on $\{ x= 2^{k-1} \}.$ Thus, on  $\{ x= 2^{k-1} \}$ we have that
$$g_k=a_k u_k ,~ \frac{\partial g_k}{\partial y}=a_k \frac{\partial u_k}{\partial y},$$
and the claim follows by applying Lemma \ref{Lem:Zeros_u} again. We will pick $m_k$ so that these zeros are not destroyed by perturbations bellow $m_k.$ More precisely, we inductively define $\{a_i \},\{m_i\}$ so that for all $i\geq 1$ the following conditions are satisfied\footnote{For $i=1$, 1) is just $0<a_1\leq A_1.$}

\begin{itemize}
\item[1)] $0<a_i\leq \min \left(A_i , \frac{\min(m_1,\ldots ,m_{i-1} )}{2^i \cdot \|u_i \|_{C^1(B_{2^i})}} \right)$ ; 
\item[2)] $m_i>0$ and for every $p:\R^2\to \R$ which satisfies $\|p \|_{C^1(B_{2^i})}<m_i$, the perturbation $g_i+p$ still has exactly $c_i-1$ zeros in $\{ 2^{i-1} \} \times [-2^{i-1},0]$, which all belong to $\{ 2^{i-1} \} \times (-2^{i-1},0)$, and at each of them $\frac{\partial (g_i+p)}{\partial y}\neq 0$.
\end{itemize}

Assume that $a_1,\ldots , a_k, m_1,\ldots,m_k$ satisfying 1) and 2) have been defined. We choose $a_{k+1}$ to be an arbitrary number which satisfies condition 1), i.e.
$$0<a_{k+1}\leq \min \left(A_{k+1} , \frac{\min(m_1,\ldots ,m_{k} )}{2^{k+1} \cdot \|u_{k+1} \|_{C^1(B_{2^{k+1}})}} \right).$$
Since $a_{k+1}>0$, the above considerations show that $m_{k+1}$ which satisfies 2) exists. This completes the inductive definition of $\{a_i\}$ and $\{m_i\}.$

To finish the proof of the proposition,  we show that $g=\sum_{i=1}^\infty a_i u_i$ has the required properties.  Firstly, by condition 1), for all $i\geq 1$, $a_i\leq A_i$ and thus $g$ is well-defined. Now, fix an arbitrary integer $k\geq 1.$ By 1), we have that for all $i\geq k+1$
$$a_i \leq  \frac{\min(m_1,\ldots ,m_{i-1} )}{2^i \cdot \|u_i \|_{C^1(B_{2^i})}}  \leq \frac{m_k}{2^i \cdot \|u_i \|_{C^1(B_{2^k})}} $$
and thus
$$\Big\Vert\sum_{i=k+1}^\infty a_i u_i \Big\Vert_{C^1(B_{2^k})} \leq \sum_{i=k+1}^\infty a_i\| u_i \|_{C^1(B_{2^k})} \leq \sum_{i=k+1}^\infty \frac{m_k}{2^i} <m_k.$$
By property 2) of $m_k$, $g=g_k+ \sum_{i=k+1}^\infty a_i u_i $ has exactly $c_k-1$ zeros in $\{ 2^{k-1} \} \times (-2^{k-1},0)$ and at each of them $\frac{\partial g}{\partial y}\neq 0.$ This finishes the proof.
\end{proof}

\begin{proof}[Proof of Theorem \ref{Thm:Main}]
Let $c_k=n_k+1$, $k\geq 1$ and let $A_1,A_2,\ldots$ be given by Proposition \ref{prop:modulus}. For $k\geq 1$, let $\{a_k\}, 0<a_k\leq A_k$, be a sequence given by Proposition \ref{prop:zeros}. Proposition \ref{prop:modulus}, implies that $g=\sum_{k=1}^\infty a_k u_k$ is well-defined, harmonic and $\mu(g,r)\leq e^{r^{1+\varepsilon}}.$ We define
$$h(x,y)= (g(x,y), \sin (\pi x) e^{\pi y}).$$
It follows that
$$\mu(h,r)=\max_{|(x,y)|\leq r} \sqrt{g(x,y)^2+ (\sin (\pi x) e^{\pi y})^2}\leq \sqrt{e^{r^{1+\varepsilon}}+e^{2\pi r}}=O(e^{r^{1+\varepsilon}}),  $$
which proves the first part of the theorem. On the other hand, $\sin (\pi x) e^{\pi y}=0$ if and only if $x\in \Z$ and hence zeros of $h$ coincide with zeros of $g$ on the lines $\{ x\in \Z \}.$ By Proposition \ref{prop:zeros}, for each $k\geq 1$, there are exactly $n_k=c_k-1$ of them inside $\{2^{k-1}\} \times (-2^{k-1},0).$ We conclude that $h$ has at least $\sum_{i=1}^k n_i$ isolated zeros inside $B_{2^k}$, which finishes the proof.
\end{proof}

\begin{rem}\label{rem:regular_zeros}
The condition $\frac{\partial g}{\partial y}\neq 0$ in Proposition \ref{prop:zeros}, guarantees that the zeros of $h$ which we constructed are not critical points of $h.$ To see this, notice that the matrix of derivatives of $h$ is
$$dh= \begin{pmatrix} \frac{\partial g}{\partial x} & \frac{\partial (\sin (\pi x) e^{\pi y}) }{\partial x} \\ \frac{\partial g}{\partial y} & \frac{\partial (\sin (\pi x) e^{\pi y})}{\partial y} \end{pmatrix}=\begin{pmatrix} \frac{\partial g}{\partial x} & -\pi\cos (\pi x) e^{\pi y} \\ \frac{\partial g}{\partial y} & \pi\sin (\pi x) e^{\pi y} \end{pmatrix},$$
and that $ \sin (\pi x) e^{\pi y}=0$, $\cos (\pi x) e^{\pi y}\neq0$ when $x\in \Z.$
\end{rem}

\section{Proof of Corollary \ref{Coro:Holomorphic}}
We will need the following classical result.
\begin{thrm}\label{Holomorhpic_Extension}
Let $h:\R^2 \to \R^2$ be a harmonic map. There exists a constant $C>0$ and a holomorphic map $f:\C^2\to \C^2$, such that $f|_{\R^2}=h$ and for all $r\geq 0$
\begin{equation}\label{Modulus_Increase}
\mu(f,r) \leq C \mu (h,2r).
\end{equation}
\end{thrm}

The proof of Theorem \ref{Holomorhpic_Extension} follows, for example, from \cite[Lemma 3]{Hayman70}.

\begin{proof}[Proof of Corollary \ref{Coro:Holomorphic}]
Let $h:\R^2 \to \R^2$ be the map given by Theorem \ref{Thm:Main} such that $\mu(h,r)=O(e^{r^{1+\varepsilon/2}}).$ Let $f:\C^2 \to \C^2$ be its holomorphic extension given by Theorem \ref{Holomorhpic_Extension}. By (\ref{Modulus_Increase}) we have that
$$\mu(f,r)\leq C \mu (h,2r) = O(e^{(2r)^{1+\varepsilon/2}})=O(e^{r^{1+\varepsilon}}).$$
On the other hand, since $f|_{\R^2}=h$, all zeros of $h$ are also zeros of $f.$ We wish to show that the isolated zeros given by Theorem \ref{Thm:Main} remain isolated as zeros of $f$ in $\C^2.$ None of these zeros are critical points of $h$, see Remark \ref{rem:regular_zeros}.  The complex derivatives of $f$ at these zeros coincide with the real derivatives of $h.$ Indeed, since $f$ is holomorphic, the complex and real derivatives of $f$ coincide and we use $f|_{\R^2}=h.$ Hence, all the zeros remain regular points of $f$ and thus isolated in $\C^2.$
\end{proof}

\end{document}